\documentclass{amsart}
\usepackage{url}
\usepackage{graphicx}
\usepackage{color}
\usepackage{times}
\usepackage{cite}
\usepackage{enumerate,latexsym}
\usepackage{amsmath,amssymb}
\usepackage{amsthm}
\usepackage{verbatim}
\usepackage{mathrsfs}
\usepackage{tabularx}

\newtheorem{thm}{Theorem}[section]
\newtheorem*{theorem*}{Theorem}
\newtheorem*{acknowledgement*}{Acknowledgement}
\newtheorem{cor}[thm]{Corollary}

\newtheorem{lem}[thm]{Lemma}
\newtheorem{prop}[thm]{Proposition}
\theoremstyle{definition}
\newtheorem{defn}[thm]{Definition}
\theoremstyle{remark}
\newtheorem{rem}[thm]{Remark}

\numberwithin{equation}{section}

\newcommand{\set}[1]{\left\{#1\right\}}
\newcommand{\Real}{\mathbb R}

\newcommand{\func}[1]{\ensuremath{\mathop{\mathrm{#1}}} }

\newcommand{\spt}[0]{\func{spt}}

\newcommand{\sing}[0]{\func{sing}}
\newcommand{\reg}[0]{\func{reg}}
\newcommand{\xX}[0]{\mathbf{x}}

\newcommand{\yY}[0]{\mathbf{y}}

\newcommand{\cC}[0]{\mathcal{C}}

\newcommand{\OO}{\mathbf{0}}
\newcommand{\sstar}{\star\star}

\title{Closed hypersurfaces of low entropy in $\Real^4$ are isotopically trivial}

\author{Jacob Bernstein}
\address{Department of Mathematics, Johns Hopkins University, 3400 N. Charles Street, Baltimore, MD 21218}
\email{bernstein@math.jhu.edu}

\author{Lu Wang}
\address{Department of Mathematics, California Institute of Technology, 1200 E. California Boulevard, CA 91125}
\email{drluwang@caltech.edu}

\begin{document}

\begin{abstract}
We show that any closed connected hypersurface in $\mathbb{R}^4$ with entropy less than or equal to that of the round cylinder is smoothly isotopic to the standard three-sphere.
\end{abstract}

\maketitle

\section{Introduction} \label{IntroSec}
If $\Sigma$ is a {hypersurface}, that is, a smooth properly embedded codimension-one submanifold of $\Real^{n+1}$, then its \emph{Gaussian surface area} is
\begin{equation}
F[\Sigma]=\int_{\Sigma}\Phi\, d\mathcal{H}^{n}=(4\pi)^{-\frac{n}{2}}\int_{\Sigma}e^{-\frac{|\xX|^2}{4}} d\mathcal{H}^n,
\end{equation}
where $\mathcal{H}^n$ is $n$-dimensional Hausdorff measure. Following Colding--Minicozzi \cite{CMGen}, define the \emph{entropy} of $\Sigma$ to be
\begin{equation}
\lambda[\Sigma]=\sup_{\yY\in\Real^{n+1}, \rho>0}F[\rho\Sigma+\yY].
\end{equation}
Throughout the paper, let $\mathbb{S}^n$ be the standard $n$-sphere in $\mathbb{R}^{n+1}$ and $\mathbb{S}^{n-1}\times\mathbb{R}$ the round cylinder in $\mathbb{R}^{n+1}$. The main result of the paper is the following.

\begin{thm}\label{MainTopThm}
If $\Sigma\subset\Real^{4}$ is a closed (i.e., compact without boundary) connected hypersurface with $\lambda[\Sigma]\leq\lambda[\mathbb{S}^2\times\Real]$, then $\Sigma$ is smoothly isotopic to $\mathbb{S}^3$.
\end{thm}

\begin{rem}
Theorem \ref{MainTopThm} was announced in \cite{BWUniqueExpander}. While we were finishing the writing of this paper, we learned about work of Chodosh-Choi-Mantoulidis-Schulze \cite{CCMS} on generic mean curvature flow which provides an alternative approach to Theorem \ref{MainTopThm}.
\end{rem}

Entropy is a natural geometric quantity that measures complexity and is invariant under rigid motion and dilations. It is known that $\lambda[\mathbb{R}^n]=1$ and, by a computation of Stone \cite{Stone},
\begin{equation}
2>\lambda[\mathbb{S}^1]>\frac{3}{2}>\lambda[\mathbb{S}^2]>\ldots>\lambda[\mathbb{S}^n]>\ldots\rightarrow\sqrt2.
\end{equation}
In \cite{BWInvent}, we prove a conjecture of Colding-Ilmanen-Minicozzi-White \cite[Conjecture 0.9]{CIMW} (cf. \cite{KetoverZhou, JZhu}) that, for $2\leq n\leq 6$, the entropy of a closed hypersurface in $\Real^{n+1}$ is uniquely (modulo translations and dilations) minimized by $\mathbb{S}^n$. We further show, in \cite[Corollary 1.3]{BWDuke} and \cite[Theorem 1.1]{BWGT}, that, for $n=2$ or $3$, any closed connected hypersurface $\Sigma\subset\Real^{n+1}$ with $\lambda[\Sigma]\leq \lambda[\mathbb{S}^{n-1}\times\mathbb{R}]$ is diffeomorphic to $\mathbb{S}^n$. By Alexander's theorem \cite{Alexander}, any surface in $\Real^3$ that is topologically a two-sphere is isotopic to $\mathbb{S}^2$. The analogous question for a three-sphere in $\Real^4$ -- known as the Schoenflies problem -- is a major open problem; see \cite{Scharlemann, Scharlemann2} for partial results in which this conjecture is proved for hypersurfaces whose embeddings are ``simple" in certain topological senses. Theorem \ref{MainTopThm} may be thought of as an affirmative answer to the Schoenflies problem for hypersurfaces that are ``simple" in a geometric sense, namely, for those that have low entropy.

In \cite{BWGT}, we also studied the topology of low entropy closed hypersurfaces in higher dimensions.  In particular, strong evidence was provided that, for $n\geq 4$, any closed connected hypersurface $\Sigma\subset \Real^{n+1}$ with $\lambda[\Sigma]\leq \lambda[\mathbb{S}^{n-1}\times \Real]$ is a homology $\mathbb{S}^n$. However, this result is necessarily conditional as it required a better understanding than current existences of low entropy minimal cones and low entropy self-shrinkers in higher dimensions. Using the more detailed analysis of this paper we are able to improve the conclusions of this conditional result to their sharpest possible form.

In order to state the conditional result, first let $\mathcal{S}_n$ be the set of self-shrinkers in $\mathbb{R}^{n+1}$, that is $\Sigma\in\mathcal{S}_n$ if and only if $\Sigma$ is a hypersurface in $\mathbb{R}^{n+1}$ satisfying 
\begin{equation} \label{ShrinkerEqn}
\mathbf{H}_\Sigma+\frac{\mathbf{x}^\perp}{2}=\mathbf{0}
\end{equation}
where $\mathbf{x}$ is the position vector, the superscript $\perp$ denotes the projection to the unit normal $\mathbf{n}_\Sigma$ of $\Sigma$, and $\mathbf{H}_\Sigma=-H_\Sigma\mathbf{n}_\Sigma=-(\mathrm{div}_\Sigma\mathbf{n}_\Sigma)\mathbf{n}_\Sigma$ is the mean curvature vector. Let $\mathcal{S}_n^*$ be the set of non-flat elements of $\mathcal{S}_n$ -- these are precisely the models of how singularities of mean curvature flow form. For any $\Lambda>0$, let 
$$
\mathcal{S}_n(\Lambda)=\set{\Sigma\in \mathcal{S}_n\colon  \lambda[\Sigma]<\Lambda} \mbox{ and } \mathcal{S}_n^*(\Lambda)=\mathcal{S}^*_n \cap \mathcal{S}_n(\Lambda).
$$
Next, let $\mathcal{RMC}_n$ denote the space of \emph{regular minimal cones} in $\Real^{n+1}$, that is $\mathcal{C}\in \mathcal{RMC}_n$ if and only if it is a proper subset of $\Real^{n+1}$ and $\mathcal{C}\backslash\set{\OO}$ is a hypersurface in $\Real^{n+1}\backslash\set{\OO}$ that is invariant under dilation about $\OO$ and with vanishing mean curvature. Let $\mathcal{RMC}_n^*$ denote the set of non-flat elements of $\mathcal{RMC}_n$ -- i.e., cones whose second fundamental forms do not identically vanish. For any $\Lambda>0$, let 
$$
\mathcal{RMC}_n(\Lambda)=\set{\mathcal{C}\in \mathcal{RMC}_n : \lambda[\mathcal{C}]< \Lambda} \mbox{ and } \mathcal{RMC}_n^*(\Lambda)=\mathcal{RMC}^*_n \cap \mathcal{RMC}_n(\Lambda).
$$

Let us now fix a dimension $n\geq 3$ and a value $\Lambda>1$. The first hypothesis is 
\begin{equation} \label{Assump1}
\mbox{For all $3\leq k\leq n$, }\mathcal{RMC}_{k}^*(\Lambda)=\emptyset \tag{$\star_{n,\Lambda}$}.
\end{equation}
Observe that all regular minimal cones in $\mathbb{R}^2$ consist of unions of rays and so $\mathcal{RMC}^*_1=\emptyset$. As great circles are the only geodesics in $\mathbb{S}^2$, $\mathcal{RMC}_2^*=\emptyset$. The second hypothesis is
\begin{equation} \label{Assump2}
\mathcal{S}_{n-1}^*(\Lambda) =\emptyset. \tag{$\sstar_{n,\Lambda}$}
\end{equation}
Obviously this holds only if $\Lambda\leq\lambda[\mathbb{S}^{n-1}]$. Denote by $\lambda_n=\lambda[\mathbb{S}^n]$. We then show the following conditional result in general dimensions.

\begin{thm}\label{MainCondThm}
Fix $n\geq 3$ and $\Lambda\in (\lambda_{n}, \lambda_{n-1}]$. If \eqref{Assump1} and \eqref{Assump2} both hold and $\Sigma$ is a closed connected hypersurface in $\Real^{n+1}$ with $\lambda[\Sigma]\leq\Lambda$, then $\Sigma$ is smoothly isotopic to $\mathbb{S}^n$.
\end{thm}

\begin{rem} 
By the results of \cite{BWInvent} and \cite{JZhu}, there does \emph{not} exist a closed hypersurface $\Sigma$ so that $\lambda[\Sigma]\leq\lambda_n$ unless $\Sigma$ is a round sphere. Thus, we require $\Lambda>\lambda_n$ in order to make Theorem \ref{MainCondThm} non-trivial. 
\end{rem}

\begin{rem}
When $n=3$, Marques-Neves' proof of the Willmore conjecture \cite[Theorem B]{MarquesNeves} and our earlier result \cite[Corollary 1.2]{BWDuke} ensure that ($\star_{3,\lambda_2}$) and ($\sstar_{3,\lambda_2}$) hold. Thus, Theorem \ref{MainTopThm} is a corollary of Theorem \ref{MainCondThm}.
\end{rem}

As in our previous work \cite{BWInvent, BWDuke, BWGT}, the tool we use is the mean curvature flow. Specifically, we use a weak formulation of mean curvature flow (see \cite{CGG}, \cite{ES1,ES2, ES3, ES4} and \cite{IlmanenElliptic}). What is new in the current paper is that we now use a careful analysis of low entropy self-expanders of mean curvature flow carried out in \cite{BWUniqueExpander} -- which in turn builds on work done in \cite{BWBanach, BWProper, BWIntDegree, BWExpanderRelEnt, BWMinMax} -- in order to understand how certain singularities of the flow resolve. In previous work, we considered only properties of low entropy self-shrinkers -- i.e., we analyzed only how singularities formed.

Recall, that a \emph{mean curvature flow} is a one-parameter family of hypersurfaces $\Sigma_t\subset\mathbb{R}^{n+1}$ that satisfies
\begin{equation}\label{MCF}
\left(\frac{\partial\mathbf{x}}{\partial t}\right)^\perp=\mathbf{H}_{\Sigma_t}.
\end{equation}
A \emph{self-expander} is a hypersurface $\Gamma\subset\mathbb{R}^{n+1}$ satisfying
\begin{equation}
\mathbf{H}_{\Gamma}-\frac{\mathbf{x}^\perp}{2}=\mathbf{0}.
\end{equation}
For a self-expander $\Gamma$ the family $\set{\sqrt{t}\, \Gamma}_{t>0}$ is an immortal solution to the mean curvature flow while for a self-shrinker $\Sigma$, i.e., a solution to \eqref{ShrinkerEqn}, the family $\set{\sqrt{-t}\, \Sigma}_{t<0}$ is an ancient solution. A fundamental property of the mean curvature flow is that the flow starting from any closed initial hypersurfaces develops a singularity in finite time and that for many such initial hypersurfaces the flow does not disappear at this singularity and, instead, can be continued as a weak flow.

Important for our applications is that, by Huisken's monotonicity formula \cite{Huisken}, the entropy is monotone non-increasing under mean curvature flow and that singularities of the flow are modeled on self-shrinkers -- see also \cite{IlmanenSing}. As shown in \cite{BWGT}, under the hypotheses of Theorem \ref{MainTopThm}, the flow starting from a closed connected hypersurface in $\Real^4$ with small entropy will develop only asymptotically conical singularities or closed singularities before its extinction time and eventually shrinks to round points -- a similar fact is true conditionally and so applies to Theorem \ref{MainCondThm}. In fact we will show the only closed singularity occurs at the extinction time -- see Remark \ref{MainThmPfRem}.
 
By our previous work \cite{BWExpanderRelEnt}, see also \cite{AIC} and \cite{EHAnn}, self-expanders model the behavior of a flow when it emerges from a conical singularity. While Huisken's monotonicity formula implies the tangent flow is backwardly self-similar, there is currently no known reason for the forward in time behavior of the tangent flow to be that of a self-expander.  Instead, we use a forward monotonicity formula from \cite{BWExpanderRelEnt} and take a second blowup to obtain a self-expanding flow.  This is the source of certain technical difficulties because singularities may accumulate into the past. To handle this, we use a bubble-tree blowup argument familiar from other areas of geometric analysis.  Specifically, we combine such a blowup argument with \cite{BWUniqueExpander} to show the flow passing through asymptotically conical singularities is smooth away from a negligible set of times and, moreover, stays within the same isotopy class whenever it is smooth -- i.e., the isotopy class does not change as one crosses any intermediate singular time. The theorem then follows from this easily. 
 
\subsection*{Acknowledgements}
The first author was partially supported by the NSF Grants DMS-1609340 and DMS-1904674 and the Institute for Advanced Study with funding provided by the Charles Simonyi Endowment. The second author was partially supported by the NSF Grants DMS-2018221(formerly DMS-1811144) and DMS-2018220 (formerly DMS-1834824), the funding from the Wisconsin Alumni Research Foundation and a Vilas Early Career Investigator Award by the University of Wisconsin-Madison, and a von Neumann Fellowship by the Institute for Advanced Study with funding from the Z\"{u}rich Insurance Company and the NSF Grant DMS-1638352. 

\section{Preliminaries} \label{PrelimSec}
In this section, we fix notation for the rest of the paper and recall some background on mean curvature flow.  Experts should feel free to consult this section only as needed.

\subsection{Basic notions} \label{NotionSec}
Here is a list of notation that we use throughout the paper. 

\begin{tabularx}{0.95\textwidth}{lX}
$B^n_R(p)$ & the open ball in $\mathbb{R}^{n}$ centered at $p$ with radius $R$; \\
$\bar{B}^n_R(p)$ & the closed ball in $\mathbb{R}^{n}$ centered at $p$ with radius $R$; \\
$\overline{U}$ & the closure of a set $U$; \\
$\partial U$ & the topological boundary of a set $U$; \\
$\nabla_{\Sigma}$ & the covariant derivative on a Riemannian manifold $\Sigma$.
\end{tabularx}

We will omit the superscript, $n$, the dimension of a ball when it is clear from the context. We will also omit the center of a ball when it is the origin. We will omit the subscript, $\Sigma$, in the covariant derivative when it is clear from the context. 

\subsection{Weak mean curvature flow} \label{WeakMCFSec}
In \cite{Brakke}, Brakke introduces a measure-theoretic weak notion of mean curvature flow, called \emph{Brakke flow}. We use the (slightly stronger) notion introduced by Ilmanen \cite[Definition 6.3]{IlmanenElliptic}, that is a family of Radon measures in $\mathbb{R}^{n+1}$ satisfying a certain variational inequality.

For a Brakke flow $\mathcal{K}=\set{\mu_t}$, a point $(\mathbf{x}_0,t_0)\in \mathbb{R}^{n+1}\times\mathbb{R}$ and $\rho>0$, let  
$$
\mathcal{K}^{(\mathbf{x}_0,t_0),\rho}=\set{\mu_t^{(\mathbf{x}_0,t_0),\rho}}
$$
where each $\mu_t^{(\mathbf{x}_0,t_0),\rho}$ is a Radon measure in $\mathbb{R}^{n+1}$ given by
$$
\mu_t^{(\mathbf{x}_0,t_0),\rho}(U)=\rho^{-n} \mu_{t_0+\rho^{2}t}(\rho U+\mathbf{x}_0) \mbox{ for any measurable set $U$}.
$$
It is readily checked that $\mathcal{K}^{(\mathbf{x}_0,t_0),\rho}$ is also a Brakke flow. Combining the monotonicity formula \cite{Huisken} and compactness result \cite{Brakke} (see also \cite{IlmanenElliptic}), Ilmanen shows the following.

\begin{prop}[{\cite[Lemma 8]{IlmanenSing}}] \label{TangentFlowProp}
Given an integral Brakke flow $\mathcal{K}=\set{\mu_t}_{t\in (t_1,t_2)}$ with bounded area ratios, a point $(\mathbf{x}_0,t_0)\in \spt(\mathcal{K})$ with $t_0>t_1$, and a sequence $\rho_i\to 0$, there is a subsequence $\rho_{i_j}$ and an integral Brakke flow $\mathcal{T}=\set{\nu_t}_{t\in\mathbb{R}}$ so that $\mathcal{K}^{(\mathbf{x}_0,t_0),\rho_{i_j}}\to\mathcal{T}$ as Brakke flows and, moreover, $\mathcal{T}$ is backward self-similar with respect to parabolic scaling about $(\mathbf{0},0)\in\mathbb{R}^{n+1}\times\mathbb{R}$, that is, $\nu_t^{(\mathbf{0},0),\rho}=\nu_t$ for all $t<0$ and $\rho>0$, and the associated varifold $V_{\nu_{-1}}$ is the critical point of the Gaussian surface area $F$.

Such $\mathcal{T}$ is called a tangent flow to $\mathcal{K}$ at $(\mathbf{x}_0,t_0)$ and denote the set of all these tangent flows by $\mathrm{Tan}_{(\mathbf{x}_0,t_0)}\mathcal{K}$.
\end{prop}

A feature of Brakke flows is that they may suddenly vanish. To overcome this, Ilmanen \cite{IlmanenElliptic} introduces a notion called \emph{matching motion}, $(\mathcal{K},\tau)$ where $\mathcal{K}$ is an integral Brakke flow and $\tau$ is an $(n+1)$-current, and uses it to synthesize the Brakke flow and the level set flow (see \cite{CGG} and \cite{ES1,ES2,ES3,ES4}) as long as the latter does not fatten. As the current $\tau$ will not be used in the proof, we will omit it. Of particular importance is that S. Wang \cite[Theorem 3.5]{W} proves a compactness theorem for matching motions with entropy less than $2$.

We collect some useful facts from our previous work \cite{BWGT, BWProper}. The first is several compactness results. A hypersurface $\Sigma$ is \emph{asymptotically conical} if $\lim_{\rho\to 0^+} \rho\Sigma=\mathcal{C}$ in $C^\infty_{loc}(\mathbb{R}^{n+1}\setminus\set{\mathbf{0}})$ where $\mathcal{C}$ is a regular cone in $\mathbb{R}^{n+1}$. When this occurs denote by $\mathcal{C}=\mathcal{C}(\Sigma)$ the asymptotic cone of $\Sigma$ and by $\mathcal{L}(\Sigma)=\mathcal{C}(\Sigma)\cap\mathbb{S}^n$ the link of $\mathcal{C}(\Sigma)$. Let $\mathcal{ACH}_n$ be the set of all asymptotically conical hypersurfaces in $\mathbb{R}^{n+1}$.

\begin{prop} \label{CompactnessProp}
Fix $n\geq 3$ and $\Lambda\in (\lambda_{n},\lambda_{n-1}]$ and assume that \eqref{Assump1} and \eqref{Assump2} hold. For any $\epsilon_0\in (0,\Lambda)$, the following is true:
\begin{enumerate}
\item The set $\mathcal{ACS}_n[\Lambda-\epsilon_0]=\set{\Sigma\in\mathcal{ACH}_n\colon \mbox{$\Sigma$ is a self-expander with $\lambda[\Sigma]\leq \Lambda-\epsilon_0$}}$ is compact in $C^\infty_{loc}(\mathbb{R}^{n+1})$;
\item The set $\mathcal{L}_n[\Lambda-\epsilon_0]=\set{\mathcal{L}(\Sigma)\colon \Sigma\in\mathcal{ACS}_n[\Lambda-\epsilon_0]}$ is compact in $C^\infty(\mathbb{S}^{n})$;
\item The set $\mathcal{E}_n[\Lambda-\epsilon_0]=\set{\Sigma\in\mathcal{ACH}_n\colon\mbox{$\Sigma$ is a self-expander with $\lambda[\Sigma]\leq \Lambda-\epsilon_0$}}$ is compact in $C^\infty_{loc}(\mathbb{R}^{n+1})$.
\end{enumerate}
\end{prop}

\begin{proof}
The first and second claim are, respectively, Corollary 3.4 and Proposition 3.5 of \cite{BWGT}. The last is Theorem 1.1 of \cite{BWProper}.
\end{proof}

The next proposition summarizes some properties about tangent flows of low entropy.

\begin{prop} \label{CondTangentFlowProp}
Fix $n\geq 3$ and $\Lambda\in (\lambda_{n},\lambda_{n-1}]$ and assume that \eqref{Assump1} and \eqref{Assump2} hold. If $\mathcal{T}=\set{\nu_t}_{t\in\mathbb{R}}$ is a matching motion in $\mathbb{R}^{n+1}$ such that $\nu_{-1}=\mathcal{H}^n\lfloor\Sigma$ for $\Sigma\in\mathcal{S}_n(\Lambda)$, then the following is true:
\begin{enumerate}
\item $\Sigma$ is either smoothly isotopic to $\mathbb{S}^n$ or asymptotically conical;
\item If $\Sigma$ is asymptotically conical and $\lambda[\Sigma]\leq\Lambda-\epsilon_0$ for some $\epsilon_0>0$, then there is a radius $R_0=R_0(n,\Lambda,\epsilon_0)>1$ and a constant $C_0=C_0(n,\Lambda,\epsilon_0)>0$ so that for each $|t|\leq 1$ there is a function $v_t\colon\mathcal{C}(\Sigma)\setminus B_{R_0}\to \mathbb{R}$ satisfying
$$
\sup_{\mathcal{C}(\Sigma)\setminus B_{R_0}} \sum_{i=0}^2 |\mathbf{x}|^{i+1} |\nabla^i v_t| \leq C_0
$$
and so that
$$
\spt(\nu_t)\setminus B_{2R_0}\subseteq\set{\mathbf{x}(p)+v_t(p)\mathbf{n}_{\mathcal{C}(\Sigma)}(p)\colon p\in\mathcal{C}(\Sigma)\setminus B_{R_0}} \subseteq\spt(\nu_t).
$$
\end{enumerate}
\end{prop}

\begin{proof}
The first claim follows from \cite[Theorem 0.7]{CIMW}\footnote{Although the result states for diffeomorphisms, the proof indeed gives smooth isotopies.} and \cite[Proposition 3.3]{BWGT}. The second is Corollary 3.6 of \cite{BWGT} for $t<0$ while, for $t>0$, follows from Proposition \ref{CompactnessProp}, the pseudo-locality \cite{INS} and the interior regularity \cite{EHInvent}.
\end{proof}

\subsection{Isotopies and related concepts} \label{IsotopySec}
We say two smooth embeddings $\mathbf{f}_0, \mathbf{f}_1\colon M\to \Real^{n+1}$ are \emph{isotopic} if there is a continuous map $\mathbf{F}\colon [0,1]\to C^\infty (M; \Real^{n+1})$ so that $\mathbf{F}(0)=\mathbf{f}_0$, $\mathbf{F}(1)=\mathbf{f}_1$ and, for each $\tau\in [0,1]$, $\mathbf{F}(\tau)$ is an embedding. Two hypersurfaces $\Sigma_0, \Sigma_1\subset \Real^{n+1}$ are \emph{isotopic} if there exist smooth embeddings $\mathbf{f}_0\colon M\to \Sigma_0$ and $\mathbf{f}_1\colon M\to \Sigma_1$ so that $\mathbf{f}_0$ and $\mathbf{f}_1$ are isotopic.

Fix a $\delta\in (0,1)$. Two hypersurfaces $\Sigma_0, \Sigma_1\subset B_{4R}(p)$ are \emph{$\delta$-isotopic} if there are smooth embeddings $\mathbf{f}_0\colon M\to \Sigma_0$ and $\mathbf{f}_1\colon M\to \Sigma_1$ and a continuous map $\mathbf{F}\colon [0,1]\to C^\infty(M; \Real^{n+1})$ so that 
\begin{enumerate}
\item $\mathbf{F}(0)=\mathbf{f}_0$ and $\mathbf{F}(1)=\mathbf{f}_1$;
\item $\mathbf{F}(\tau)$ is an embedding for each $\tau\in [0,1]$;
\item $\mathbf{F}(\tau)\circ \mathbf{f}_0^{-1}(B_R(p))\subset B_{2R}(p)$ for each $\tau\in [0,1]$;
\item For each $\tau\in [0,1]$,
$$
\sup_{\Sigma_0\cap (B_{4R}(p)\backslash B_{R}(p))} \sum_{i=0}^1 R^{i-1}\left| \nabla^i_{\Sigma_0} (\mathbf{F}(\tau)\circ \mathbf{f}_0^{-1})-\nabla^i_{\Sigma_0}\mathbf{x}|_{\Sigma_0}\right|\leq \delta.
$$
\end{enumerate}

Fix a unit vector $\mathbf{e}$, a point $\mathbf{x}_0\in\mathbb{R}^{n+1}$ and $r,h>0$. Let 
$$
C_{\mathbf{e}}(\mathbf{x}_0,r,h)=\set{\mathbf{x}\in\mathbb{R}^{n+1}\colon |(\mathbf{x}-\mathbf{x}_0)\cdot\mathbf{e}|<h, |\mathbf{x}-\mathbf{x}_0|^2<r^2+|(\mathbf{x}-\mathbf{x}_0)\cdot\mathbf{e}|^2}
$$
be the solid open cylinder with axis $\mathbf{e}$ centered at $\mathbf{x}_0$ and of radius $r$ and height $2h$. A hypersurface $\Sigma$ is a \emph{$C^2$ $\mathbf{e}$-graph of size $\delta$ on scale $r$ at $\mathbf{x}_0$} if there is a function $f\colon B^n_r\subset P_{\mathbf{e}}\to\mathbb{R}$ with
$$
\sum_{i=0}^2 r^{i-1} \Vert\nabla^i f\Vert_{C^0}< \delta,
$$
where $P_\mathbf{e}$ is the $n$-dimensional subspace of $\mathbb{R}^{n+1}$ normal to $\mathbf{e}$, so that 
$$
\Sigma\cap C_{\mathbf{e}}(\mathbf{x}_0, r, \delta r)=\set{\mathbf{x}_0+\mathbf{x}(x)+f(x)\mathbf{e}\colon x\in B_r^n}.
$$

\begin{lem} \label{IsotopyGluingLem}
Let $B_{4r_1}(p_1),\ldots, B_{4r_J}(p_J)$ be pairwise disjoint open balls in $\mathbb{R}^{n+1}$ and assume $r_j\geq r_1$ for each $1\leq j\leq J$. Let $\Sigma_0\subset\mathbb{R}^{n+1}$ be a  hypersurface that is a $C^2$ $\mathbf{n}_{\Sigma_0}$-graph of size $1$ on scale $c_0r_1$ at every $p\in \Sigma_0\setminus\bigcup_{j=1}^J B_{r_1}(p_j)$. Then there is a sufficiently small $\delta_0=\delta_0(n,c_0)>0$ so that if a hypersurface $\Sigma_1\subset\mathbb{R}^{n+1}$ satisfies:
\begin{enumerate}
\item $\Sigma_0\cap B_{4r_j}(p_j)$ is $\delta_0$-isotopic to $\Sigma_1 \cap B_{4r_j}(p_j)$ for each $1\leq j\leq J$;
\item There is a continuous family of functions $u_\tau\colon \Sigma_0\setminus\bigcup_{j=1}^J B_{r_1}(p_j)\to\mathbb{R}$ for $\tau\in [0,1]$ with $u_\tau=0$ and
$$
\sup_{ \Sigma_0\setminus \bigcup_{j=1}^J B_{r_1}(p_j)} \sum_{i=0}^1 r_1^{i-1} |\nabla^i_{\Sigma_0} u_\tau |\leq \delta_0
$$
and so that
$$
\Sigma_1\backslash \bigcup_{j=1}^J B_{2r_1}(p_j)\subseteq\set{\mathbf{x}(p)+u_1(p)\mathbf{n}_{\Sigma_0}(p) \colon p\in \Sigma_0\backslash \bigcup_{j=1}^J B_{r_1}(p_j)}\subseteq\Sigma_1,
$$
\end{enumerate}
then there is an isotopy $\mathbf{F}\colon [0,1]\to C^\infty(\Sigma_0;\mathbb{R}^{n+1})$ with $\mathbf{F}(0)=\mathbf{x}|_{\Sigma_0}$ and $\mathbf{F}(1)(\Sigma_0)=\Sigma_1$ and so that, for each $\tau\in [0,1]$, 
\begin{enumerate}
\item $\mathbf{F}(\tau)(B_{4r_j}(p_j))\subset B_{8r_j}(p_j)$ for each $1\leq j\leq J$;
\item $\mathbf{F}(\tau)(p)=\mathbf{x}(p)+ u_{\ell(\tau)}(p)\mathbf{n}_{\Sigma_0}(p)$ for $p\in\Sigma_0\setminus\bigcup_{j=1}^J B_{4r_j}(p_j)$ where $\ell\colon [0,1]\to [0,1]$ is a continous function.
\end{enumerate}
\end{lem}

\begin{proof}
Define $\mathbf{F}_0(\tau)\colon \Sigma_0\backslash \bigcup_{j=1}^J B_{r_1}(p_j)\to\mathbb{R}^{n+1}$ by 
$$
\mathbf{F}_0(\tau)=\mathbf{x}(p) + u_\tau (p)\mathbf{n}_{\Sigma_0}(p).
$$
By the hypotheses on $\Sigma_0$ and $u_\tau$, for $\delta_0$ sufficiently small $\mathbf{F}_0(\tau)$ is an embedding for each $\tau\in [0,1]$.

Set $U_0=\mathbb{R}^{n+1}\setminus\bigcup_{j=1}^J \bar{B}_{3r_j}(p_j)$ and $U_j=B_{4r_j}(p_j)$ for $1\leq j \leq J$. Let $\set{\phi_j}_{0\leq j\leq J}$ be a partition of unity subordinate to the open cover $\set{U_j}_{0\leq j\leq J}$ so $|\nabla\phi_j|\leq 2r_1^{-1}$ for each $j$. For $\tau\in [0,1]$, we define $\tilde{\mathbf{F}}(\tau)\colon \Sigma_0\to \Real^{n+1}$ by
$$
\tilde{\mathbf{F}}(\tau)=\sum_{j=0}^J \phi_j \mathbf{F}_j(\tau)
$$
where $\mathbf{F}_j$ for $1\leq j\leq J$ are the $\delta_0$-isotopies with $\mathbf{F}_j(0)=\mathbf{x}|_{\Sigma_0\cap U_j}$ that the hypotheses ensure exist.

Up to shrinking $\delta_0$, one has that $\tilde{\mathbf{F}}(\tau)$ is an embedding for every $\tau\in [0,1]$, $\tilde{\mathbf{F}}(0)=\mathbf{x}|_{\Sigma_0}$ and $\Sigma_1^\prime=\tilde{\mathbf{F}}(1)(\Sigma_0)$ is sufficiently close, in the $C^1$ topology, to $\Sigma_1$ so they are isotopic. Combining this isotopy with $\tilde{\mathbf{F}}$ gives the desired isotopy between $\Sigma_0$ and $\Sigma_1$.
\end{proof}

We will need a notion of a.c.-isotopies between asymptotically conical hypersurfaces which is closely related to those introduced in \cite[Section 2]{BWUniqueExpander}. Fix an asymptotically conical hypersurface $\Gamma\subset\mathbb{R}^{n+1}$ with asymptotic cone $\mathcal{C}=\mathcal{C}(\Gamma)$. Thus, for some $R>1$ large enough, $\pi_{\mathcal{C}}$ -- the nearest point projection onto $\mathcal{C}$ -- restricts to a diffeomorphism of $\Gamma_R=\Gamma\setminus\bar{B}_R$ onto image and denote its inverse by $\theta_{\Gamma_R}$. For any integer $k\geq 2$, let $\mathcal{ACE}^k_n(\Gamma)$ be the space of $C^k$-asymptotically conical $C^k$ embeddings of $\Gamma$ into $\mathbb{R}^{n+1}$, i.e., $C^k$ embeddings $\mathbf{g}\colon\Gamma\to\mathbb{R}^{n+1}$ so that $\lim_{\rho\to 0^+} \rho (\mathbf{g}\circ\theta_{\Gamma_R})(\rho^{-1} p)=\mathbf{h}(p)$ in $C^k_{loc}(\mathcal{C}\setminus\set{\mathbf{0}})$ where $\mathbf{h}\colon\mathcal{C}\to \mathbb{R}^{n+1}$ is a homogeneous of degree one (i.e., $\rho\mathbf{h}(\rho^{-1} p)=\mathbf{h}(p)$ for all $p\in\mathcal{C}$ and $\rho>0$) $C^k $ embedding and denote by $\mathrm{tr}_\infty^1[\mathbf{g}]=\mathbf{h}|_{\mathcal{C}\cap\mathbb{S}^n}$. Equip the space $\mathcal{ACE}^k_n(\Gamma)$ with the $C^k_1$ norm 
$$
\Vert\mathbf{g}\Vert_{C^k_1}=\sup_{p\in\Gamma} \sum_{i=0}^k (1+|\mathbf{x}(p)|)^{1-i} |\nabla^i \mathbf{g}|.
$$
We then let $\mathcal{ACE}_n(\Gamma)=\bigcap_{k\geq 2} \mathcal{ACE}^k_n(\Gamma)$ with the usual Fr\'{e}chet topology.

Two elements $\mathbf{g}_0,\mathbf{g}_1\in\mathcal{ACE}_n(\Gamma)$ are \emph{a.c.-isotopic} if there exists a continuous path $\mathbf{G}\colon [0,1]\to \mathcal{ACE}^\infty_n(\Gamma)$ with $\mathbf{G}(0)=\mathbf{g}_0$ and $\mathbf{G}(1)=\mathbf{g}_1$. Two asymptotically conical hypersurfaces $\Gamma_0,\Gamma_1\subset\mathbb{R}^{n+1}$ are \emph{a.c.-isotopic} if there are two elements $\mathbf{g}_j\in\mathcal{ACE}_n(\Gamma)$ with $\mathbf{g}_j(\Gamma)=\Gamma_j$ for $j\in\set{0,1}$ so that they are a.c.-isotopic. By composing with $\mathbf{g}_0^{-1}$, we will always take $\Gamma=\Gamma_0$ and $\mathbf{G}(0)=\mathbf{x}|_{\Gamma_0}$. Furthermore, for fixed $R>1$ and $C>0$ we say $\mathbf{G}$ is \emph{$(R,C)$-regular} if there exists a continuous family of functions $v_\tau\colon\mathcal{C}(\Gamma_0)\setminus B_{R}\to \mathbb{R}$ for $\tau\in [0,1]$ satisfying
$$
\sup_{\mathcal{C}(\Gamma_0)\setminus B_{R}}\sum_{i=0}^2 |\mathbf{x}|^{i+1} |\nabla^i v_\tau| \leq C
$$
and so that
$$
\mathbf{G}(\tau)(\Gamma_0)\setminus B_{2R}\subseteq\set{\mathbf{x}(p)+v_\tau(p)\mathbf{n}_{\mathcal{C}(\Gamma_0)}(p))\colon p\in\mathcal{C}(\Gamma_0)}\subseteq \mathbf{G}(\tau)(\Gamma_0).
$$
In this case, we also call $\Gamma_0$ and $\Gamma_1=\mathbf{G}(\tau)(\Gamma_0)$ are $(R,C)$-regular a.c.-isotopic. 

\begin{lem} \label{ACIsotopyLem}
Fix a regular cone $\mathcal{C}\subset\mathbb{R}^{n+1}$ that is a $C^2$ $\mathbf{n}_{\mathcal{C}}$-graph of size $1$ on scale $r_0$ at every $p\in \mathcal{L}=\mathcal{C}\cap\mathbb{S}^n$. Let $\Gamma_0$ and $\Gamma_1$ be two asymptotically conical hypersurfaces in $\mathbb{R}^{n+1}$ with $\mathcal{C}(\Gamma_0)=\mathcal{C}(\Gamma_1)=\mathcal{C}$ and that are $(R_0,C_0)$-regular a.c.-isotopic for some constants $R_0>1$ and $C_0>0$. Then for every $\delta\in (0,1)$ there is a radius $R_1=R_1(n, r_0,R_0,C_0,\delta)>1$ so that, for any $R>R_1$, $\Gamma_0\cap B_{4R}$ is $\delta$-isotopic to $\Gamma_1\cap B_{4R}$.
\end{lem}

\begin{proof}
Suppose $\mathbf{G}\colon[0,1]\to\mathcal{ACE}_n(\Gamma_0)$ is an $(R_0,C_0)$-regular a.c.-isotopy between $\Gamma_0$ and $\Gamma_1$. Fix a $\kappa\in (0,1)$ to be determined later in the proof. Let $\Pi_{\mathcal{C}}(p)$ be the nearest point projection (in $\partial B_{|p|}$) of $p$ onto $\mathcal{C}\cap \partial B_{|p|}$. Let $\Gamma_\tau=\mathbf{G}(\tau)(\Gamma_0)$. Then there is a radius $\tilde{R}_1=\tilde{R}_1(n,r_0,R_0,C_0,\delta,\kappa)>2R_0$ so that, for each $\tau\in [0,1]$, $\Pi_{\mathcal{C}}$ restricts to a diffeomorphism of $\Gamma_\tau\setminus B_{\tilde{R}_1}$ onto $\mathcal{C}\setminus B_{\tilde{R}_1}$ and its inverse $\mathbf{f}_\tau$ satisfies 
$$
\sup_{\mathcal{C}\setminus B_{\tilde{R}_1}} \sum_{i=0}^2 |\mathbf{x}|^{i-1} |\nabla^i \mathbf{f}_{\tau}-\nabla^i\mathbf{x}|_{\mathcal{C}}| \leq \kappa\delta.
$$
By the continuity of $\mathbf{G}$ there is a radius $\tilde{R}_0>2\tilde{R}_1$, which may depend on the isotopy $\mathbf{G}$, so that for each $\tau\in [0,1]$
\begin{itemize}
\item $\mathbf{G}(\tau)(\Gamma_0\cap B_{2\tilde{R}_0})\subset B_{4\tilde{R}_0}$;
\item $\sup_{\Gamma_0\setminus B_{\tilde{R}_0}} \sum_{i=0}^1 |\mathbf{x}|^{i-1} \left| \nabla^i \mathbf{G}(\tau)-\nabla^i \mathbf{x}|_{\Gamma_0} \right| \leq \kappa\delta$.
\end{itemize}

Let $\phi\colon\mathbb{R}^{n+1}\to [0,1]$ be a radial cutoff so that $\phi=1$ outside $B_{2\tilde{R}_1}$ and $\phi=0$ in $B_{\tilde{R}_1}$. Let $\mathbf{V}_{\tau}=\phi |\mathbf{x}|^2 |\mathbf{x}^T|^{-2} \mathbf{x}^T$ be a vector field on $\Gamma_\tau$ whose flow dilates on $\Gamma_\tau$. If $\set{\Phi_\tau(t)}_{t\in\mathbb{R}}$ is the family of diffeomorphisms of $\Gamma_\tau$ generated by $\mathbf{V}_\tau$, then
$$
\frac{\partial}{\partial t}|\Phi_\tau (t,p)|=\phi(\Phi_\tau (t,p))|\Phi_\tau (t,p)|.
$$ 
Thus, for $t\geq 0$,
$$
|\mathbf{x}(p)| \leq |\Phi_{\tau}(t,p)| \leq e^t |\mathbf{x}(p)|
$$
with equality in the second inequality for $p\in\Gamma_\tau\setminus B_{2\tilde{R}_1}$ so $\Phi_\tau(\Gamma_\tau \cap \partial B_R)=\Gamma_\tau \cap \partial B_{e^{t}R}$ and $\Phi_\tau(\Gamma_\tau\cap B_R)=\Gamma_\tau \cap B_{e^t R}$ for $R>2\tilde{R}_1$. Likewise, for $t<0$,
$$
e^t |\mathbf{x}(p)| \leq  |\Phi_{\tau}(t,p)| \leq |\mathbf{x}(p)|,
$$
with equality in the first equality as long as $p\in \Gamma_\tau\setminus B_{2\tilde{R}_1 e^{-t}}$.
Now define 
$$
\tilde{\mathbf{G}}(\tau)=\Phi_\tau (\log(\tilde{R}_1\tilde{R}_0^{-1})) \circ \mathbf{G}(\tau) \circ \Phi_0 (\log(\tilde{R}_1^{-1}\tilde{R}_0)).
$$
It is readily checked that, for an appropriate choice of $\kappa=\kappa(n)$ and corresponding $\tilde{R}_1$, as $\log(\tilde{R}_1^{-1}\tilde{R}_0)>0$ one has
\begin{itemize}
\item $\tilde{\mathbf{G}}(\tau)(\Gamma_0)=\Gamma_\tau$ for each $\tau\in [0,1]$;
\item $\tilde{\mathbf{G}}(\tau)(\Gamma_0\cap B_{2\tilde{R}_1})\subset \Gamma_\tau \cap B_{4\tilde{R}_1}$ for each $\tau\in [0,1]$;
\item $\sup_{\Gamma_0\setminus B_{2\tilde{R}_1}} \sum_{i=0}^1 |\mathbf{x}|^{i-1} |\nabla^i \tilde{\mathbf{G}}(\tau)-\nabla^i \mathbf{x}|_{\Gamma_0}| \leq \delta$ for each $\tau\in [0,1]$.
\end{itemize}
Hence the result follows with $R_1=2\tilde{R}_1$.
\end{proof}

\section{Basic tangent flow analysis} \label{BasicTangentFlowSec}
In this section we combine the forward monotonicity formula for flows coming out of a cone and trapped between two expanders from \cite{BWExpanderRelEnt} with the main theorem of \cite{BWUniqueExpander} to prove an initial structural result for model tangent flows of low entropy. Backwards in time these flows will be self-shrinkers and so the emphasis is on the forward in time behavior. 

We first need a lemma showing that expander mean convex solutions coming out of a cone can be rescaled to produce self-expanders.

\begin{lem} \label{ExpanderMCLem}
Fix $n\geq 3$ and $\Lambda\in (\lambda_n,\lambda_{n-1}]$ and assume that \eqref{Assump1} and \eqref{Assump2} hold. Let $\set{\Sigma_{t}}_{t>0}$ be a mean curvature flow of connected asymptotically conical hypersurfaces in $\Real^{n+1}$ so that
$$
\lim_{t\to 0^+}\mathcal{H}^n\lfloor \Sigma_t=\mathcal{H}^n \lfloor \cC
$$
where $\mathcal{C}$ is a regular cone in $\mathbb{R}^{n+1}$ and, for a consistent choice of unit normal, $\mathbf{n}_{\Sigma_t}$, of $\Sigma_t$, 
$$
E^{O,t}_{\Sigma_t}=2t H_{\Sigma_t}+{\mathbf{x}\cdot\mathbf{n}_{\Sigma_t}} <0.
$$
If $\lambda[\Sigma_t]<\Lambda$ for all $t$, then
$$
\lim_{t\to 0^+} t^{-1/2} \Sigma_t = \Gamma \mbox{ in $C^{\infty}_{loc}(\mathbb{R}^{n+1})$}
$$
where $\Gamma$ is an asymptotically conical self-expander with $\cC(\Gamma)=\cC$. In fact, $\Gamma$ is a.c.-isotopic to $\Sigma_1$.  Finally, if $U_t$ are components of $\Real^{n+1}\setminus \Sigma_t$ so $\mathbf{n}_{\Sigma_t}$ points out of $U_t$ and $U_\infty$ is the corresponding component of $\Real^{n+1}\setminus \Gamma$, then one has $U_\infty\subset t_1^{-1/2} U_{t_1}\subset t_{2}^{-1/2} U_{t_2}$ for $0<t_1<t_2$. 
\end{lem}

\begin{proof}
This is essentially proved in \cite[Proposition 5.1]{BWUniqueExpander} for $t\to \infty$.  The only difference when $t\to 0$ is that the limit will, generally, not be stable. This is only relevant for the regularity of the limit and of the convergence. Instead, one may appeal to the entropy bound and \eqref{Assump1} and \eqref{Assump2} and Brakke's regularity theorem \cite{Brakke} -- see also \cite{WhiteReg} -- to obtain the desired regularity for the limit and for the convergence.
\end{proof}

\begin{prop} \label{BasicRegProp}
Fix $n\geq 3$, $\Lambda\in (\lambda_n,\lambda_{n-1}]$ and $\epsilon_0\in (0,\Lambda)$. Assume that \eqref{Assump1} and \eqref{Assump2} hold. Let $\mathcal{T}=\set{\nu_t}_{t\in \Real}$ be a matching motion in $\Real^{n+1}$ such that $\nu_{-1}=\mathcal{H}^n \lfloor \Sigma$ for $\Sigma\in\mathcal{ACS}_n[\Lambda-\epsilon_0]$. Then there is a $\rho_+=\rho_+(\mathcal{T})\in (0,1)$ so that
\begin{enumerate}
\item $\mathcal{T}\lfloor \Real^{n+1}\times (0, \rho_+^2)$ is a smooth flow; 
\item For all $t\in (0, \rho_+^2)$, the asymptotically conical hypersurfaces $\Gamma_t=\spt(\nu_t)$ are $(R_0,C_0)$-regular a.c.-isotopic to $\Sigma$, where $R_0=R_0(n,\Lambda,\epsilon_0)$ and $C_0=C_0(n,\Lambda,\epsilon_0)$ are given by Proposition \ref{CondTangentFlowProp}.
\end{enumerate}
\end{prop}

\begin{proof}
As $\Sigma$ is a self-shrinker, it is connected by the Fraenkel property of self-shrinkers. Let $f\in C^{\infty}(\Sigma)$ be the unique positive function that satisfies
$$
\int_{\Sigma} |f|^2 e^{-\frac{|\mathbf{x}|^2}{4}} d\mathcal{H}^n =1
$$
and
$$
-L_{\Sigma} f=-\left(\Delta_\Sigma -\frac{\mathbf{x}}{2} +|A_{\Sigma}|^2 +\frac{1}{2}\right) f= \mu_0 f
$$ 
where $\mu_0<-1$ is the lowest eigenvalue of the shrinker stability operator -- see \cite[Proposition 4.1]{BWDuke} for the existence of such a function. As observed in \cite[Proposition 4.1]{BWDuke} this $f$ has sublinear growth. Hence, as $\Sigma$ is asymptotically conical, there is an $\tilde{\epsilon}>0$ so that for all $\epsilon \in (-\tilde{\epsilon}, \tilde{\epsilon})$ one has
$$
\Sigma^{\epsilon}=\set{\mathbf{x}(p)+\epsilon f(p) \mathbf{n}_{\Sigma}(p): p\in \Sigma}
$$
are embedded asymptotically conical hypersurfaces with $\cC(\Sigma^\epsilon)=\cC(\Sigma)$. Up to shrinking $\epsilon$, one can ensure that $\lambda[\Sigma^\epsilon]<\lambda[\Sigma]\leq \Lambda-\epsilon_0$ when $\epsilon\neq 0$ and each $\Sigma_\epsilon$ has shrinker mean curvature, $\mathbf{H}_{\Sigma^\epsilon}+\frac{\mathbf{x}^\perp}{2}$ that points away from $\Sigma=\Sigma^0$ for $\epsilon\neq 0$. These facts are all proved in \cite[Proposition 4.2]{BWDuke} -- though the direction of the shrinker mean curvature is not explicitly stated it is easily determined from the argument. These hypersurfaces also form a foliation around $\Sigma=\Sigma^0$. 
		
Use the choice of unit normal on $\Sigma$ to define $\Omega_-(\Sigma)$  as the component of $\Real^{n+1}\backslash \Sigma$ for which $\mathbf{n}_\Sigma$ points outward -- as $\Sigma$ is connected there are only two components. This definition extends in an obvious and compatible way to the $\Sigma^\epsilon$.  Using this, Let $U^\epsilon=\Omega_-(\Sigma^\epsilon)$.  Clearly, for $\epsilon_1<\epsilon_2$, one has $U^{\epsilon_1}\subseteq U^{\epsilon_2}$. Observe that for $\epsilon>0$ the shrinker mean curvature of $\Sigma^\epsilon$ points out of $U^\epsilon$, while for $\epsilon<0$ it points into the region.
		
For $\epsilon\neq 0$, let $\mathcal{T}^\epsilon=\set{\Sigma^\epsilon_t}_{t\in [-1,\infty)}$ be the mean curvature flow with initial data determined by $\Sigma^\epsilon$.  As proved in \cite[Proposition 4.5]{BWDuke} these flows are smooth for all time and remain asymptotic to $\cC(\Sigma)$. Moreover, the maximum principle ensures that if $\epsilon_1\neq \epsilon_2$, then $\Sigma_t^{\epsilon_1}$ is disjoint from $\Sigma_t^{\epsilon_2}$ for all $t\in [-1,\infty)$. In addition, each $\Sigma_t^\epsilon$ is a.c.-isotopic to $\Sigma$. Choose the unit normal $\mathbf{n}_{\Sigma_t^\epsilon}$ on $\Sigma_t^\epsilon$ that is compatible with the one on $\Sigma^\epsilon$. Using this normal, let $U^\epsilon_t$ be the component of $\mathbb{R}^{n+1}\setminus\Sigma^\epsilon_t$ so $\mathbf{n}_{\Sigma_t^\epsilon}$ points out of it. Clearly, as the flows of distinct values of $\epsilon$ remain disjoint, one has that, for $\epsilon_1<\epsilon_2$, $U^{\epsilon_1}_t\subseteq U^{\epsilon_2}_t$ for all $t\geq -1$.  Moreover, one has that $\spt(\nu_t)$, the support of the initial flow $\mathcal{T}$, satisfies $\spt(\nu_t)\subseteq \overline{U_t^{\epsilon_+}}\setminus U_{t}^{\epsilon_-}$ for any $-\tilde{\epsilon}<\epsilon_-<0<\epsilon_+<\tilde{\epsilon}$. 
	
By the compactness of Brakke flows \cite{Brakke} (see also \cite{IlmanenElliptic} and \cite{W}), one can consider $\mathcal{T}^+=\lim_{\epsilon \downarrow 0} \mathcal{T}^\epsilon$ and $\mathcal{T}^-=\lim_{\epsilon \uparrow 0} \mathcal{T}^\epsilon$. Write $\mathcal{T}^\pm=\set{\nu^\pm_t}_{t\in [-1,\infty)}$. These limits exist as the flows are topologically ordered. Indeed, let $U^+_t=\bigcap_{\epsilon\in (0,\tilde{\epsilon})} U^\epsilon_t$ and $U^-_t=\bigcup_{\epsilon\in (-\tilde{\epsilon},0)} U^\epsilon_t$ -- one readily checks these are sets of locally finite perimeter. By the entropy bound one has $\nu^\pm_t=\mathcal{H}^n \lfloor \partial^* U^\pm_t$.  Moreover, by the uniqueness for smooth mean curvature flows of bounded curvature, one has that $\mathcal{T}^\pm=\mathcal{T}$ in $\Real^{n+1}\times [-1,0)$ and, hence, in $\Real^{n+1}\times [-1,0]$.  However, in general, $\mathcal{T}^+$, $\mathcal{T}^-$ and $\mathcal{T}$ can be disjoint in $\Real^{n+1}\times (0,\infty)$.  Nevertheless, they are ordered in the expected way. Indeed, $U^-_t\subseteq U^+_t$ and $\spt(\nu_t)\subseteq \Omega_t=\overline{U^+_t}\setminus U^-_t$.
	
As each $\Sigma^\epsilon$ is shrinker mean convex for $\epsilon\neq 0$, it follows from the parabolic maximum principle (e.g., \cite[Proposition 4.4]{BWDuke}) that $\Sigma_t^\epsilon$ is also shrinker mean convex for $t\in [-1,0)$ and is expander mean convex for $t\in (0,\infty)$. Note that for $t\in (-1,0)$ the shrinker mean curvature points away from $\Omega_t=\sqrt{-t}\, \Sigma$, while, for $t>0$, the expander mean curvature points toward $\Omega_t$. The strict maximum principle and \eqref{Assump1} and \eqref{Assump2} and standard blowup arguments implies that for $t>0$ either $\spt(\nu_t^+)=\sqrt{t}\, \Gamma^+$ for $\Gamma^+$ a smooth self-expander or $\spt(\nu_t^+)=\Sigma_t^+$ where $\set{\Sigma_t^+}_{t>0}$ is a smooth mean curvature flow that is strictly expander mean convex and satisfies the other hypotheses of Lemma \ref{ExpanderMCLem}. If the later case occurs, then Lemma \ref{ExpanderMCLem} implies $\lim_{t\to 0} t^{-1/2} \Sigma_t^+=\Gamma^+$ for $\Gamma^+$ a smooth self-expander. A similar argument produces a self-expander $\Gamma^-$ corresponding to $\mathcal{T}^-$. In either case, for $t>0$,  $\Omega_t$ lies between $\sqrt{t}\, \Gamma^-$ and $\sqrt{t}\,\Gamma^+$.  As $\Gamma^\pm$ are smooth and are either limits of $\Sigma^\epsilon_{1}$ or are a.c.-isotopic to these limits, they are both a.c.-isotopic to $\Sigma$.
	
Now consider any tangent flow $\mathcal{T}^\prime$ to $\mathcal{T}$ at $(\mathbf{0},0)$. As this flow is trapped between $\sqrt{t}\, \Gamma^-$ and $\sqrt{t}\, \Gamma^+$ for $t>0$, one can appeal to the forward monotonicity formula \cite[Theorem 6.1]{BWExpanderRelEnt} and \eqref{Assump1} to see that, for $t>0$, $\mathcal{T}^\prime=\set{\sqrt{t}\, \Gamma^\prime}_{t>0}$ for $\Gamma^\prime$ a smooth self-expander trapped between $\Gamma^-$ and $\Gamma^+$. By the main result of \cite{BWUniqueExpander}, $\Gamma^\prime$ is a.c.-isotopic to $\Gamma^-$ and hence to $\Sigma$.  
	
One proves the existence of $\rho_+$ by contradiction. To see the first claim, suppose there was no such $\rho_+$, then there would be a sequence of singular points of $\mathcal{T}$, $(\mathbf{x}_i,t_i)\in\mathbb{R}^{n+1}\times (0,\infty)$ with $t_i\to 0$. By Brakke's regularity theorem (e.g., \cite{Brakke} and \cite{WhiteReg}) $(\mathbf{x}_i,t_i)\to (\mathbf{0},0)$. Let $r_i=(|\mathbf{x}_i|^2+t_i)^{1/2}$ so $r_i\to 0$ and $(\mathbf{x}_i^\prime,t_i^\prime)=(r_i^{-1}\mathbf{x}_i,r_i^{-2} t_i)$. By what we have already shown, the rescaled flows $\mathcal{T}^{(\mathbf{0},0),r_i}$, up to passing to a subsequence, converge to a tangent flow $\mathcal{T}^\prime$ which is $\sqrt{-t}\, \Sigma$ for $t<0$ and $\sqrt{t}\, \Gamma^\prime$ for $t>0$ for $\Gamma^\prime$ a smooth self-expander. As $|\mathbf{x}_i^\prime|^2+t_i^\prime=1$, up to passing to a further subsequence, $(\mathbf{x}_i^\prime,t_i^\prime)\to (\mathbf{x}^\prime_0,t^\prime_0)$ and, by the upper semicontinuity of Gaussian density, $(\mathbf{x}^\prime_0,t^\prime_0)$ lies on the support of $\mathcal{T}$. As $|\mathbf{x}^\prime_0|^2+t^\prime_0=1$ and $t^\prime_0\geq 0$, $(\mathbf{x}^\prime_0,t^\prime_0)$ is a regular point of $\mathcal{T}$ and so, invoking Brakke's regularity again, for all large $i$ the $(\mathbf{x}^\prime_i,t_i^\prime)$ are regular points of $\mathcal{T}^{(\mathbf{0},0), r_i}$. That is, the $(\mathbf{x}_i,t_i)$ are regular points of $\mathcal{T}$. This contradiction proves the first claim.

To see the second claim, again suppose there was no such $\rho_+$, then there would be a sequence $t_i>0$ with $t_i\to 0$ so that the $\spt(\nu_{t_i})$ are not $(R_0,C_0)$-regular a.c.-isotopic to $\Sigma$. Consider the rescaled flows $\mathcal{T}^{(\mathbf{0},0), \sqrt{t_i}}=\set{\nu_{t}^i}_{t\in\mathbb{R}}$ so, up to passing to a subsequence, they converge to a tangent flow $\mathcal{T}^\prime=\set{\nu_t^\prime}_{t\in\mathbb{R}}$. As remarked before, $\spt(\nu_t^\prime)=\sqrt{-t}\, \Sigma$ for $t<0$ and $\spt(\nu_t^\prime)=\sqrt{t}\, \Gamma^\prime$ for $\Gamma^\prime$ a self-expander that is a.c.-isotopic to $\Sigma$. Let $\delta_0=\delta_0(\Sigma)\in (0,1)$ be the number given by Lemma \ref{IsotopyGluingLem}. There is a radius $\tilde{R}>1$ so that, for any $R>\tilde{R}$, $\Gamma^\prime\cap B_{4R}$ is $\frac{\delta_0}{2}$-isotopic to $\Sigma\cap B_{4R}$. As $\spt(\nu_1^i)\to \Gamma^\prime$ in $C^\infty_{loc}(\mathbb{R}^{n+1})$, for all large $i$ the $\spt(\nu_1^i)\cap B_{8R}$ are $\delta_0$-isotopic to $\Sigma\cap B_{8R}$. That is, $\spt(\nu_{t_i})\cap B_{8R\sqrt{t_i}}$ is $\delta_0$-isotopic to $(\sqrt{t_i}\, \Sigma)\cap B_{8R\sqrt{t_i}}$. Thus, by Proposition \ref{CondTangentFlowProp} and the pseudo-locality \cite[Theorem 1.5]{INS}, one can appeal to Lemma \ref{IsotopyGluingLem} to patch this isotopy and the flow $\mathcal{T}$ restricted to $(\mathbb{R}^{n+1}\setminus B_{2R\sqrt{t_i}})\times [-t_i,t_i]$. This shows that $\spt(\nu_{t_i})$ is $(R_0,C_0)$-regular a.c.-isotopic to $\sqrt{t_i}\, \Sigma$. As, via the flow $\mathcal{T}$, $\sqrt{t_i} \, \Sigma$ is $(R_0,C_0)$-regular a.c.-isotopic to $\Sigma$, so is $\spt(\nu_{t_i})$. This is a contradiction and completes the proof.
\end{proof}

\section{Almost isotopies} \label{AlmostIsotopySec}
In this section we show that if all the tangent flows of a mean curvature flow of low entropy are ``almost isotopies" in a certain sense, then the flow itself is an almost isotopy.

First of all given a matching motion $\mathcal{K}$ let $\sing(\mathcal{K})\subseteq\spt(\mathcal{K})$ be the set of singular points of $\mathcal{K}$ and $\mathrm{reg}(\mathcal{K})=\spt(\mathcal{K})\setminus\sing(\mathcal{K})$ be the set of regular points. We then let
$$
\mathrm{ST}(\mathcal{K})=\set{t_0\in \Real\colon \mbox{$(\mathbf{x}_0, t_0)\in \sing(\mathcal{K})$ for some $\mathbf{x}_0\in\mathbb{R}^{n+1}$}}
$$
be the set of singular times.
\begin{defn}
Let $\mathcal{K}=\set{\mu_t}_{t\in [-1,1)}$ be a matching motion with $\spt(\mu_t)$ compact. We call $\mathcal{K}$ an \emph{almost isotopy} if 
\begin{enumerate}
\item $\mathrm{ST}(\mathcal{K})\subseteq (-1,1)$ has $\mathcal{L}^1$ measure zero;
\item For every $t\notin\mathrm{ST}(\mathcal{K})$ either $\spt(\mu_t)=\emptyset$ or $\spt(\mu_t)$ is isotopic to $\spt(\mu_{-1})$.
\end{enumerate}
\end{defn}

\begin{defn}
Let $\mathcal{K}=\set{\mu_t}_{t\in \mathbb{R}}$ be a matching motion such that $\mu_{-1}=\mathcal{H}^n\lfloor \Sigma$ for $\Sigma$ an asymptotically conical self-shrinker in $\mathbb{R}^{n+1}$. We call $\mathcal{K}$  an \emph{almost a.c.-isotopy} if 
\begin{enumerate}
\item $\mathrm{ST}(\mathcal{K})\cap [0,1)$ has $\mathcal{L}^1$ measure zero;
\item For every $t\in [0,1)\setminus \mathrm{ST}(\mathcal{K})$, $\spt(\mu_t)$ is a.c.-isotopic to $\Sigma$.
\end{enumerate}
For fixed $R>1$ and $C>0$, an almost a.c.-isotopy $\mathcal{K}$ is called \emph{$(R,C)$-regular} if $\spt(\mu_t)$ is $(R,C)$-regular a.c.-isotopic to $\Sigma$ for every $t\notin\mathrm{ST}(\mathcal{K})$.
\end{defn}

Define the distance $d$ on space-time $\mathbb{R}^{n+1}\times\mathbb{R}$ to be 
$$
d((\mathbf{x},t),(\mathbf{y},s))=\sqrt{|\mathbf{x}-\mathbf{y}|^2+|t-s|}.
$$
Denote by $B^d_R((\mathbf{x}_0,t_0))$ the (open) ball in the metric $d$ centered at $(\mathbf{x}_0,t_0)$ with radius $R$. Given a matching motion $\mathcal{K}=\set{\mu_t}$ and a point $X_0=(\mathbf{x}_0,t_0)\in\reg(\mathcal{K})$, let 
$$
R^{\mathcal{K}}_{reg}(X_0)=\sup\set{r>0\colon\mbox{$\Sigma_{t_0}=\spt(\mu_{t_0})$ is a $C^2$ $\mathbf{n}_{\Sigma_{t_0}}$-graph of size $1$ on scale $r$ at $\mathbf{x}_0$}}
$$
be the regularity radius of $\mathcal{K}$ at $X_0$. We omit the superscript, $\mathcal{K}$, when it is clear from the context.

\begin{lem} \label{RegRadiusLem}
Let $\mathcal{K}=\set{\mu_t}_{t\in [-1,1)}$ be a matching motion in $\mathbb{R}^{n+1}$ and $X_0=(\mathbf{x}_0,t_0)\in\spt(\mathcal{K})$. Suppose every $\mathcal{T}=\set{\nu_t}_{t\in\mathbb{R}}\in\mathrm{Tan}_{X_0}\mathcal{K}$ satisfies $\nu_{-1}=\mathcal{H}^n\lfloor\Sigma$ for $\Sigma$ an asymptotically conical self-shrinker \footnote{By recent work of Chodosh-Schulze \cite{CS}, the asymptotically conical multiplicity-one tangent flow is unique.}. Then there exists $\rho_0=\rho_0(\mathcal{K},X_0)>0$ so that 
\begin{enumerate}
\item $\mathcal{K}\lfloor (B^d_{\rho_0}(X_0)\cap\set{t<t_0})$ is a smooth flow;
\item $\kappa_0=\inf\set{d(X,X_0)^{-1} R_{reg}(X)\colon X\in \reg(\mathcal{K})\cap B^d_{\rho_0}(X_0)\cap\set{t<t_0}}>0$.
\end{enumerate}
\end{lem}

\begin{proof}
We argue by contradiction. Suppose there was no such $\rho_0$, then there would be a sequence of points $X_i=(\mathbf{x}_i,t_i)\in\spt(\mathcal{K})$ with $t_i<t_0$ and $d(X_i,X_0)\to 0$ and so that one of the following situations occurs:
\begin{enumerate}
\item $X_i\in\sing(\mathcal{K})$ for all large $i$;
\item $X_i\in\reg(\mathcal{K})$ and $d(X_i,X_0)^{-1} R_{reg}(X_i)\to 0$.
\end{enumerate}
Let $r_i=d(X_i,X_0)>0$ and $\tilde{X}_i=(r_i^{-1}(\mathbf{x}_i-\mathbf{x}_0),r_i^{-2} (t_i-t_0))$ so $d(\tilde{X}_i,O)=1$ where $O=(\mathbf{0},0)\in\mathbb{R}^{n+1}\times\mathbb{R}$. Up to passing to a subsequence, one may assume $\tilde{X}_i\to\tilde{X}_0$ with $d(\tilde{X}_0,O)=1$. Consider the rescaled flows $\mathcal{K}^{X_0,r_i}$ and, up to passing to a further subsequence, they converge to a tangent flow $\mathcal{T}=\set{\nu_t}_{t\in\mathbb{R}}$ where $\nu_{-1}=\mathcal{H}^n\lfloor\Sigma$ for $\Sigma$ an asymptotically conical self-shrinker. As $\tilde{X}_i\in\spt(\mathcal{K}^{X_0,r_i})$ for all large $i$, by the upper semi-continuity of Gaussian density, one has $\tilde{X}_0\in\spt(\mathcal{T})$. As $\tilde{X}_0\neq O$ and $\mathcal{T}\lfloor((\mathbb{R}^{n+1}\setminus\set{\mathbf{0}})\times (-\infty,0])$ is a smooth flow, it follows that $\tilde{X}_0\in\reg(\mathcal{T})$. By Brakke's regularity theorem \cite{Brakke}, for all large $i$, $\tilde{X}_i\in\reg(\mathcal{K}^{X_0,r_i})$ and $R_{reg}(\tilde{X}_i)>\kappa>0$. It follows that, for all large $i$, $X_i\in\reg(\mathcal{K})$ and $d(X_i,X_0)^{-1} R_{reg}(X_i)>\kappa>0$. That is, neither of the situations occurs and this is a contradiction, finishing the proof.
\end{proof}

Combining Proposition \ref{CondTangentFlowProp} and Lemma \ref{RegRadiusLem} we obtain the following corollary.

\begin{cor} \label{IsolateSingTimeCor}
Fix $n\geq 3$ and $\Lambda\in (\lambda_n,\lambda_{n-1}]$ and assume that \eqref{Assump1} and \eqref{Assump2} hold. Let $\mathcal{K}=\set{\mu_t}_{t\in [-1,1)}$ be a matching motion in $\mathbb{R}^{n+1}$ with  $\lambda[\mu_{-1}]\leq \Lambda$ and assume that $\spt(\mu_{-1})$ is either a closed hypersurface or an asymptotically conical self-shrinker. Then the following is true:
\begin{enumerate}
\item For every $t_0\in (-1,1)$, $\sing_{t_0}(\mathcal{K})=\set{\mathbf{x}_0\colon (\mathbf{x}_0,t_0)\in\sing(\mathcal{K})}$ is a finite set;
\item For every $t_0\in\mathrm{ST}(\mathcal{K})$, there is a $\Delta_0>0$ so that $(t_0-\Delta_0,t_0)\cap\mathrm{ST}(\mathcal{K})=\emptyset$.
\end{enumerate}
\end{cor}

\begin{lem} \label{DensityLem}
Fix $n\geq 3$, $\Lambda\in (\lambda_n,\lambda_{n-1}]$, $\epsilon_0\in (0,\Lambda)$ and $\delta\in (0,1)$. Assume that \eqref{Assump1} and \eqref{Assump2} hold. Let $\mathcal{K}=\set{\mu_t}_{t\in [-1,1)}$ be a matching motion in $\mathbb{R}^{n+1}$ with $\lambda[\mu_{-1}]\leq \Lambda-\epsilon_0$. Suppose that $(\mathbf{x}_0,t_0)\in\spt(\mathcal{K})$ is such that every $\mathcal{T}\in\mathrm{Tan}_{(\mathbf{x}_0,t_0)}\mathcal{K}$ is an $(R_0,C_0)$-regular almost a.c.-isotopy. Given $\alpha>R_1$ and $\gamma\in (0,1)$, there exists $\rho_1=\rho_1(\mathcal{K},\mathbf{x}_0,t_0,\alpha,\gamma)>0$ so that: if $\rho<\rho_1$ and
$$
I_\delta^c(\rho)=\set{s\in (0,\rho^2)\colon \mbox{$\spt(\mu_{t_0+s})$ is not $2\delta$-isotopic to $\spt(\mu_{t_0-s})$ in $B_{4\alpha\rho}(\mathbf{x}_0)$}},
$$
then one has
$$
\mathcal{L}^1(I_\delta^c(\rho))\leq\gamma\rho^2.
$$
Here $R_0=R_0(n,\Lambda,\epsilon_0)$ and $C_0=C_0(n,\Lambda,\epsilon_0)$ are given by Proposition \ref{CondTangentFlowProp}, and $R_1=R_1(n,\Lambda,\epsilon_0,\delta)$ is chosen by  Proposition \ref{CompactnessProp} and Lemma \ref{ACIsotopyLem}.
\end{lem}

\begin{proof}
We argue by contradiction. That is, suppose there was no such $\rho_1$.  That means there are a sequence of $\rho_i \to 0$ so that $\mathcal{L}^1(I_\delta^c(\rho_i))>\gamma \rho_i^2$. 

Up to passing to a subsequence, one has that $\mathcal{K}^{(\mathbf{x}_0,t_0), \rho_i}$ converges to an element $\mathcal{T}=\set{\nu_t}_{t\in\mathbb{R}}\in \mathrm{Tan}_{(\mathbf{x}_0,t_0)}\mathcal{K}$ with $\nu_{-1}=\mathcal{H}^n\lfloor\Sigma$ for $\Sigma$ an asymptotically conical self-shrinker. By our hypotheses, $\mathcal{T}$ is an $(R_0,C_0)$-regular almost a.c.-isotopy. Thus, $\mathrm{ST}(\mathcal{T})\cap (0,1)$ has Lebesgue measure zero and, by Lemma \ref{ACIsotopyLem}, for every $t\in(0,1)\setminus \mathrm{ST}(\mathcal{T})$, $\spt(\nu_t)$ is $\delta$-isotopic to $\spt(\nu_{-t})$ in $B_{4\alpha}$. However, the nature of the convergence contradicts the assumption on the size of $I_{\delta}^c(\rho_i)$. This proves the proposition.
\end{proof}

\begin{prop} \label{LocAlmostIsotopyProp}
Fix $n\geq 3$, $\Lambda\in (\lambda_n,\lambda_{n-1}]$ and $\epsilon_0\in (0,\Lambda)$. Assume that \eqref{Assump1} and \eqref{Assump2} hold. Let $\mathcal{K}=\set{\mu_t}_{t\in [-1,1)}$ be a matching motion in $\mathbb{R}^{n+1}$ with $\lambda[\mu_{-1}]\leq\Lambda-\epsilon_0$ and assume $\Sigma=\spt(\mu_{-1})$ is a closed  hypersurface. If $t_0\in (-1,1)$ is such that, for every $(\mathbf{x}_0,t_0)\in\spt(\mathcal{K})$, every $\mathcal{T}\in\mathrm{Tan}_{(\mathbf{x}_0,t_0)}\mathcal{K}$ is an $(R_0,C_0)$-regular almost a.c.-isotopy, where $R_0$ and $C_0$ are given by Proposition \ref{CondTangentFlowProp}, then there exists $\rho_2=\rho_2(\mathcal{K},t_0)>0$ so that: if $\rho<\rho_2$ and 
$$
I_{iso}^c(\rho)=\set{s\in (0, \rho^2)\colon \spt(\mu_{t_0+s}) \mbox{ is not isotopic to } \spt(\mu_{t_0-s})},
$$
then one has
$$
\mathcal{L}^1(I_{iso}^c(\rho))\leq \frac{1}{2}\rho^2.
$$

The same conclusions hold if one instead supposes  $\Sigma=\spt(\mu_{-1})$ is an asymptotically conical self-shrinker, with the set
$$
I_{iso}^c(\rho)=\set{s\in (0, \rho^2)\colon \spt(\mu_{t_0+s}) \mbox{ is not $(R_0,C_0)$-regular a.c.-isotopic to} \spt(\mu_{t_0-s})}.
$$
\end{prop}

\begin{proof}
By Lemma \ref{RegRadiusLem} and Corollary \ref{IsolateSingTimeCor}, there exist $\rho_0=\rho_0(\mathcal{K},t_0)>0$ and $c_0=c_0(\mathcal{K},t_0)>0$ so that if $X=(\mathbf{x},t)\in\spt(\mathcal{K})$ with $t\in (t_0-\rho_0,t_0)$, then $X\in \reg(\mathcal{K})$ and $R_{reg}(X)>c_0 d(X,\mathrm{sing}_{t_0}(\mathcal{K}))$. Let $\delta_0=\delta_0(n,c_0)$ be the number given by Lemma \ref{IsotopyGluingLem}. Let $N(t_0)$ be the number of elements of $\mathrm{sing}_{t_0}(\mathcal{K})=\set{\mathbf{x}_1, \ldots, \mathbf{x}_{N(t_0)}}$. Hence, it follows from Lemma \ref{DensityLem} with $\delta=\frac{\delta_0}{2}$ and $\gamma=\frac{1}{2N(t_0)}$ that for every large $\alpha$ there exists $\rho_1=\rho_1(\mathcal{K},t_0,\alpha)>0$ so that for every $\rho<\rho_1$ for all $s\in (0,\rho^2)\setminus I_{\delta,i}^c(\rho)=I_{\delta, i}(\rho)$ one has $\spt(\mu_{t_0+s})$ is $\delta$-isotopic to $\spt(\mu_{t_0-s})$ in $B_{4\alpha\rho}(\mathbf{x}_i)$ where  $(\mathbf{x}_i,t_0)\in\sing(\mathcal{K})$, $1\leq i \leq N(t_0)$. As $\alpha$ may be arbitrarily large, one uses the pseudo-locality \cite{INS} and Lemma \ref{IsotopyGluingLem} to patch these $\delta$-isotopies with the flow $\mathcal{K}$ and obtain isotopies between $\spt(\mu_{t_0+s})$ and $\spt(\mu_{t_0-s})$ for any $s\in \bigcap_{i=1}^{N(t_0)} I_{\delta, i} (\rho)$ with $\rho>0$ small.  Clearly,
$$
I_{iso}^c(\rho)\subseteq (0, \rho^2) \setminus \bigcap_{i=1}^{N(t_0)} I_{\delta, i}(\rho)=\bigcup_{i=1}^{N(t_0)} I_{\delta,i}^c(\rho)
$$
and so, by Lemma \ref{DensityLem} and the choice of $\gamma$ one has
$$
\mathcal{L}^1(I_{iso}^c(\rho))\leq \sum_{i=1}^{N(t_0)} \gamma \rho^2 =\frac{1}{2}\rho^2.
$$
 Moreover, if $\Sigma$ is an asymptotically conical self-shrinker, then Proposition \ref{CondTangentFlowProp} implies these isotopies are $(R_0,C_0)$-regular. This last observation concludes the proof.
\end{proof}

\begin{thm} \label{AlmostACIsotopyThm}
Fix $n\geq 3$, $\Lambda\in (\lambda_n,\lambda_{n-1}]$ and $\epsilon_0\in (0,\Lambda)$. Assume that \eqref{Assump1} and \eqref{Assump2} hold. Let $\mathcal{K}=\set{\mu_t}_{t\in [-1,1)}$ be a matching motion in $\mathbb{R}^{n+1}$ with $\lambda[\mu_{-1}]\leq \Lambda-\epsilon_0$ and assume $\Sigma=\spt(\mu_{-1})$ is a closed connected hypersurface (resp. $\Sigma=\spt(\mu_{-1})$ is an asymptotically conical self-shrinker). If for every $(\mathbf{x}_0,t_0)\in\spt(\mathcal{K})$ with $t_0\in (-1,1)$ (resp. $t_0\in (0,1)$), every $\mathcal{T}=\set{\nu_t}_{t\in\mathbb{R}}\in\mathrm{Tan}_{(\mathbf{x}_0,t_0)}\mathcal{K}$ satisfies either
\begin{enumerate}
\item $\spt(\nu_{-1})$ is compact; or
\item $\spt(\nu_{-1})$ is non-compact and $\mathcal{T}$ is an $(R_0,C_0)$-regular almost a.c.-isotopy, where $R_0$ and $C_0$ are given by Proposition \ref{CondTangentFlowProp},
\end{enumerate}
then $\mathcal{K}$ is an almost isotopy (resp. $\mathcal{K}$ is an $(R_0,C_0)$-regular almost a.c.-isotopy).
\end{thm}

\begin{proof}
By Corollary \ref{IsolateSingTimeCor}, every point in $\mathrm{ST}(\mathcal{K})$ has Lebesgue density at most $\frac{1}{2}$. As $\mathrm{ST}(\mathcal{K})$ is a closed set, the Lebesgue density theorem implies $\mathrm{ST}(\mathcal{K})$ has Lebesgue measure zero.

We first suppose $\Sigma=\spt(\mu_{-1})$ is a closed connected hypersurface. Without loss of generality we assume $\mathcal{K}$ does not disappear, as otherwise we restrict the flow up to the extinction time, and translate in time and do a parabolic dilatation to obtain a new flow satisfying the hypotheses that does not go extinct. Let 
$$
B=\set{t\in (-1,1)\setminus\mathrm{ST}(\mathcal{K})\colon \spt(\mu_t) \mbox{ is not isotopic to } \Sigma}.
$$
The openness of $B$ ensures that it is enough to show that $B$ has Lebesgue measure zero in order to conclude $\mathcal{K}$ is an almost isotopy. 

Let
$$
B^c=\set{t\in (-1,1)\setminus\mathrm{ST}(\mathcal{K})\colon \spt(\mu_t) \mbox{ is isotopic to } \Sigma}
$$
so $B\cup B^c=(-1,1)\setminus\mathrm{ST}(\mathcal{K})$. First we show that for any $I=(a,b)$ with $-1\leq a<b\leq 1$ and $a\in B^c$
$$
t_I=\sup\set{t\in B^c\cap I}=b.
$$ 
As $a\in B^c$, in particular $a$ is a regular time and so times near $a$ are also in $B^c$. Thus, $B^c\cap I$ is non-empty so $t_I$ is well defined and $t_I>a$. Take a sequence of times $t_i\in B^c\cap I$ so that $t_i\to t_I$. Clearly, $t_I\in\mathrm{ST}(\mathcal{K})$. By Corollary \ref{IsolateSingTimeCor}, there is a $\delta>0$ so that $(t_I-\delta,t_I)\subseteq B^c \cap I$. If $t_I<b$, then $t_I$ is not the extinction time and, as $\Sigma$ is connected by hypothesis, all tangent flows at time $t_I$ have non-compact support. Hence, by our hypotheses, it follows from Proposition \ref{LocAlmostIsotopyProp} that there is a small $\Delta>0$ so that $\spt(\mu_{t_I+\Delta})$ is isotopic to $\spt(\mu_{t_I-\Delta})$ and hence to $\Sigma$. That is, $t_I+\Delta\in B^c\cap I$ but $t_I+\Delta>t_I$. This contradicts the definition of $t_I$ and thus $t_I=b$ proving the claim.

Fix any $\epsilon>0$. As $\mathrm{ST}(\mathcal{K})$ is compact and has Lebesgue measure zero, one finds a finite cover of $\mathrm{ST}(\mathcal{K})$, $I_j=(a_j,b_j)$ for $1\leq j\leq J$ with $a_1<a_2<\cdots<a_J$, that satisfies
\begin{itemize}
\item all $a_j,b_j$ are in $(-1,1)\setminus\mathrm{ST}(\mathcal{K})$;
\item $I_j\cap I_k\neq\emptyset$ only if $|j-k|\leq 1$;
\item $I_{j}$ is not a subset of $I_{k}$ when $|j-k|=1$.
\item $\sum_{j=1}^J |I_j|<\epsilon$.
\end{itemize}
We claim, it is possible to choose all $a_j\in B^c$. Indeed, as the flow is smooth on $[-1,a_1]$, one has $a_1\in B^c$. We next consider two situations. The first situation is that $I_1\cap I_2$ is empty. By the previous claim, $b_1\in B^c$ and, so either $a_2=b_1\in B^c$ or, if $b_1<a_2$ one observes the flow is smooth on $[b_1,a_2]$ and so also concludes $a_2\in B^c$. The second situation is that $I_1\cap I_2$ is non-empty. In this case the properties of the intervals ensure $a_1<a_2<b_1<b_2$. Replace $I_2$ by $I_2^\prime=(b_1,b_2)$ in the cover to obtain an new cover $\set{I'_j=(a'_j,b'_j)}_{1\leq j\leq J}$ that satisfies the same properties as the original cover but has $a^\prime_2=b_1$ is in $B^c$. Iterate this procedure on subsequent intervals to obtain a new cover $\set{I_j^{\prime\prime}=(a_j^{\prime\prime},b_j^{\prime\prime})}_{1\leq j \leq J}$ satisfying all properties of $I_j$ and, in addition, with all $a^{\prime\prime}_j\in B^c$. Appealing to the previous claim, one also has all $b^{\prime\prime}_j\in B^c$ and, hence, 
$$
(-1,1)\setminus\bigcup_{j=1}^J I^{\prime\prime}_j\subseteq B^c.
$$
Hence, 
$$
\mathcal{L}^1(B^c)\geq 2-\sum_{j=1}^J |I_j''| \geq 2-\sum_{j=1}^J |I_j|>2-\epsilon.
$$ 
Sending $\epsilon\to 0$, gives $\mathcal{L}^1(B^c)=2$. As $B$, $B^c$, and $\mathrm{ST}(\mathcal{K})$ are pairwise disjoint and their union is $(-1,1)$, 
$$
\mathcal{L}^1(B)=2-\mathcal{L}^1(B^c)-\mathcal{L}^1(\mathrm{ST}(\mathcal{K}))=0.
$$
This proves the claim.

We now consider the case that $\Sigma=\spt(\mu_{-1})$ is an asymptotically conical self-shrinker. Observe that, by the Frankel property of self-shrinkers, $\Sigma$ is connected. The arguments are essentially same as the previous case and we only mention necessary modifications. Namely, the sets $B$ and $B^c$ are replaced by, respectively,
$$
\hat{B}=\set{t\in (0,1)\setminus\mathrm{ST}(\mathcal{K})\colon \spt(\mu_t) \mbox{ is not $(R_0,C_0)$-regular a.c.-isotopic to } \Sigma}
$$
and
$$
\hat{B}^c=\set{t\in (0,1)\setminus\mathrm{ST}(\mathcal{K})\colon \spt(\mu_t) \mbox{ is $(R_0,C_0)$-regular a.c.-isotopic to } \Sigma}.
$$
The only difference is that $\spt(\mu_0)=\mathcal{C}(\Sigma)$ is singular. This can be addressed by using Proposition \ref{BasicRegProp}, that is, for all $0<t<\rho_+^2$ the $\spt(\mu_t)$ are $(R_0,C_0)$-regular a.c.-isotopic to $\Sigma$. Hence the result follows from the previous arguments with the above modifications.
\end{proof}

\section{Entropy quantization and a bubble tree-like structure} \label{BubbleTreeSec}
We need to improve the estimates on singularities given by Proposition \ref{BasicRegProp} in order to show that any non-compact tangent flow to a low entropy flow is actually an a.c.-almost isotopy. To do this requires an iterated blowup procedure that is reminiscent of the bubble-tree structure occurring in other areas of geometric analysis.  This ultimately shows that any tangent flow to a low entropy flow is an almost a.c.-isotopy by iterated blowups.  

We first establish a gap for the entropy of cones of asymptotically conical self-shrinkers.

\begin{lem} \label{GapLem}
Let $\Sigma$ be a non-flat asymptotically conical self-shrinker. One has
$$
\lambda[\Sigma]>\lambda[\cC(\Sigma)].
$$
\end{lem}

\begin{proof}
As $\cC(\Sigma)$ is a smooth cone it follows that there is a point $\mathbf{x_0}\in\mathbb{R}^{n+1}$ so that
$$
F[\cC(\Sigma)+\xX_0]=\lambda[\cC(\Sigma)].
$$
See \cite{BWProper} for proof. It follows from Huisken's monotonicity formula \cite{Huisken} and the fact that $\Sigma$ is smooth and non-flat that
$$
\lambda[\Sigma]\geq F[\sqrt{2}\, \Sigma+\mathbf{x}_0]>F[\cC(\Sigma)+\xX_0]=\lambda[\cC(\Sigma)].
$$
This proves the claim.
\end{proof}

Using this result and a compactness result from previous work we have the following:

\begin{prop}\label{QuantaProp}
Fix $n\geq 3$, $\Lambda\in (\lambda_{n},\lambda_{n-1}]$ and $\epsilon_0\in (0,\Lambda)$. Assume that \eqref{Assump1} and \eqref{Assump2} hold. There is a $\delta_1=\delta_1(n,\Lambda,\epsilon_0)>0$ so that: if $\Sigma\in\mathcal{ACS}_n[\Lambda-\epsilon_0]$ is non-flat, then 
$$
\lambda[\Sigma]\geq \lambda[\cC(\Sigma)]+\delta_1.
$$	
\end{prop}

\begin{proof}
We argue by contradiction.  Indeed, suppose there is a sequence $\Sigma_i\in\mathcal{ACS}_n[\Lambda-\epsilon_0]$ that are non-flat and so $\lambda[\Sigma_i]\leq \lambda[\cC(\Sigma_i)]+\frac{1}{i}$.
	
By White's version \cite{WhiteReg} of Brakke's regularity theorem, $\lambda[\Sigma_i]=F[\Sigma_i]\geq 1+\epsilon(n)$ where $\epsilon(n)>0$ is some fixed constant independent of the $\Sigma_i$. By Proposition \ref{CompactnessProp}, up to passing to a subsequence, one has $\Sigma_i\to \Sigma_\infty$ in $C^{\infty}_{loc}(\Real^{n+1})$ for $\Sigma_\infty\in\mathcal{ACS}_n[\Lambda-\epsilon_0]$ and $\mathcal{L}(\Sigma_i)\to \mathcal{L}(\Sigma_\infty)$ in $C^\infty(\mathbb{S}^n)$.  As $\lambda[\Sigma_\infty]=F[\Sigma_\infty]=\lim_{i\to \infty} F[\Sigma_i]=\lim_{i\to \infty}\lambda[\Sigma_i]$, one has $\lambda[\Sigma_\infty]\geq 1+\epsilon(n)$ and so $\Sigma_\infty$ is not flat.  Moreover, $\lambda[\Sigma_\infty]\leq \lim_{i\to \infty} \lambda[\cC(\Sigma_i)]$.  However, by \cite[Lemma 3.8]{BWProper} one has $\lambda[\cC(\Sigma)]=\lim_{i \to \infty} \lambda[\cC(\Sigma_i)]$ and so
$$
\lambda[\Sigma_\infty]\leq \lambda[\cC(\Sigma_\infty)].
$$
As $\Sigma_\infty$ is not flat, this contradicts Lemma \ref{GapLem} and proves the claim.
\end{proof}

We conclude

\begin{thm} \label{BubbleTreeThm}
Fix $n\geq 3$, $\Lambda\in (\lambda_n,\lambda_{n-1}]$ and $\epsilon_0\in (0,\Lambda)$. Assume that \eqref{Assump1} and \eqref{Assump2} hold. Let $\mathcal{T}=\set{\nu_t}_{t\in\mathbb{R}}$ be a matching motion in $\Real^{n+1}$ such that $\nu_{-1}=\mathcal{H}^n \lfloor \Sigma$ for $\Sigma\in\mathcal{ACS}_n[\Lambda-\epsilon_0]$. Then $\mathcal{T}$ is an $(R_0,C_0)$-regular almost a.c.-isotopy, where $R_0=R_0(n,\Lambda,\epsilon_0)$ and $C_0=C_0(n,\Lambda,\epsilon_0)$ are given by Proposition \ref{CondTangentFlowProp}.
\end{thm}

\begin{proof}
We argue by contradiction. Suppose $\mathcal{T}$ is not an $(R_0,C_0)$-regular almost a.c.-isotopy. By Proposition \ref{CondTangentFlowProp}, all the tangent flows of $\mathcal{T}$ are either isotopic to $\mathbb{S}^n$ or asymptotically conical. Appealing to Theorem \ref{AlmostACIsotopyThm} gives a point $(\mathbf{x}_0,t_0)\in \spt(\mathcal{T})$ with $t_0\in (0,1)$ and a tangent flow $\mathcal{T}^\prime=\set{\nu_t^\prime}_{t\in\mathbb{R}}\in \mathrm{Tan}_{(\mathbf{x}_0, t_0)}\mathcal{T}$ so that $\spt(\nu_{-1}^\prime)$ is non-compact, but $\mathcal{T}^\prime$ is not an $(R_0,C_0)$-regular almost a.c.-isotopy. As $t_0>0$, Proposition \ref{QuantaProp} implies that $\lambda[\mathcal{T}^\prime]\leq \lambda[\mathcal{T}]-\delta_1$ for some uniform $\delta_1>0$.  Repeating this argument, one constructs a sequence of matching motions $\mathcal{T}^{(l)}$ each of same form as $\mathcal{T}$ --i.e., non-compact and not an $(R_0,C_0)$-regular almost a.c.-isotopy -- but with $\lambda[\mathcal{T}^{(l)}]\leq \lambda[\mathcal{T}]-\delta_1 l$. Hence, for $l$ sufficiently large one can apply White version's \cite{WhiteReg} of Brakke regularity theorem and see that all $\mathcal{T}^{(l)}$ are smooth flows and, hence, are $(R_0,C_0)$-regular almost a.c.-isotopies. This is a contradiction and proves the claim.
\end{proof}

\section{Concluding the proof} \label{SingularitySec}
We now prove Theorem \ref{MainCondThm}. Theorem \ref{MainTopThm} is an immediate consequence of this.

\begin{proof}[Proof of Theorem \ref{MainCondThm}]
Without loss of generality, we may assume that $\Sigma$ is both non-fattening and satisfies $\lambda[\Sigma]<\Lambda$. Indeed, if $\Sigma$ is (after a translation and dilation) a self-shrinker, then, by \cite[Theorem 0.7]{CIMW}, $\Sigma$ is isotopic to $\mathbb{S}^n$ and so the theorem is immediate. Otherwise, flow $\Sigma$ for a small amount of time by the mean curvature flow (using short time existence of for smooth closed initial hypersurfaces) to obtain a hypersurface, $\Sigma^\prime$, isotopic to $\Sigma$ and, by Huisken's monotonicity formula \cite{Huisken}, with $\lambda[\Sigma^\prime]<\lambda[\Sigma]$. On the one hand, if the level set flow of $\Sigma^\prime$ is non-fattening, then set $\Sigma_0=\Sigma^\prime$. On the other hand, if the level set flow of $\Sigma^\prime$ is fattening, then we can take $\Sigma_0$ to be a small normal graph over $\Sigma^\prime$ so $\lambda[\Sigma_0]<\lambda[\Sigma]$, $\Sigma_0$ is isotopic to $\Sigma$ and, because the non-fattening condition is generic, the level set flow of $\Sigma_0$ is non-fattening. We now set $\Sigma=\Sigma_0$ and we have $\Sigma$ non-fattening and $\lambda[\Sigma]<\Lambda-\epsilon_0$ for some $\epsilon_0>0$. 
	
Now consider the matching motion $\mathcal{K}=\set{\mu_t}_{t\geq 0}$ associated to $\Sigma$ -- such a motion exists by \cite[Theorem 11.4]{IlmanenElliptic}. Let $T\in (0,\infty)$ be the extinction time of $\mathcal{K}$ so $\spt(\mu_t)=\emptyset$ for all $t>T$. By hypothesis, $\Sigma$ is connected and by Proposition \ref{CondTangentFlowProp} and Theorem \ref{BubbleTreeThm}, every tangent flow $\mathcal{T}=\set{\nu_t}_{t\in\mathbb{R}}$ satisfies either $\spt(\nu_{-1})$ is compact, or $\mathcal{T}$ is an $(R_0,C_0)$-regular a.c.-isotopy. Thus, it follows from Theorem \ref{AlmostACIsotopyThm} that $\mathcal{K}$ is an almost isotopy. As singularity models at the extinction time are isotopic to $\mathbb{S}^n$, $\Sigma$ is isotopic to $\mathbb{S}^n$. 
\end{proof}
\begin{rem}\label{MainThmPfRem}
	The proof of Theorem \ref{MainCondThm} establishes that any matching motion, $\mathcal{K}$, associated to a closed connected hypersurface of low entropy is isotopic to $\mathbb{S}^n$ at every non-empty regular time of the flow.  Hence, the only compact singularity of the flow occurs at the extinction time and the flow disappears at a single spatial point at this time. Non-fattening is used only to establish the existence of such a $\mathcal{K}$.
\end{rem}


\begin{thebibliography}{99}

\bibitem{Alexander} J.W. Alexander, \emph{On the subdivision of $3$-space by a polyhedron}, Proc. Nat. Acad. Sci. 10 (1924), no. 1, 6--8.

\bibitem{AIC} S.B. Angenent, T. Ilmanen, and D.L. Chopp, \emph{A computed example of non-uniqueness of mean curvature flow in $\Real^3$}, Commun. in Partial Differential Equations 20 (1995), no. 11-12, 1937--1958. 
	
\bibitem{BWInvent} J. Bernstein and L. Wang, \emph{A sharp lower bound for the entropy of closed hypersurfaces up to dimension six}, Invent. Math. 206 (2016), no. 3, 601--627.

\bibitem{BWDuke} J. Bernstein and L. Wang, \emph{A topological property of asymptotically conical self-shrinkers of small entropy}, Duke Math. J. 166 (2017), no. 3, 403--435.

\bibitem{BWGT} J. Bernstein and L. Wang, \emph{Topology of closed hypersurfaces of small entropy}, Geom. Topol. 22 (2018), no. 2, 1109--1141.

\bibitem{BWBanach} J. Bernstein and L. Wang, \emph{The space of asymptotically conical self-expanders of mean curvature flow}, preprint. Available at \url{https://arxiv.org/abs/1712.04366}

\bibitem{BWProper} J. Bernstein and L. Wang, \emph{Smooth compactness for spaces of asymptotically conical self-expanders of mean curvature flow}, IMRN, to appear. Available at \url{https://arxiv.org/abs/1804.09076}.

\bibitem{BWIntDegree} J. Bernstein and L. Wang, \emph{An integer degree for asymptotically conical self-expanders}, preprint. Available at \url{https://arxiv.org/abs/1807.06494}.

\bibitem{BWUniqueExpander} J. Bernstein and L. Wang, \emph{Topological uniqueness for self-expanders of small entropy}, preprint.  Available at \url{https://arxiv.org/abs/1902.02642}.

\bibitem{BWExpanderRelEnt} J. Bernstein and L. Wang, \emph{Relative expander entropy in the presence of a two-sided obstacle and applications}, preprint. Available at \url{https://arxiv.org/abs/1906.07863}.

\bibitem{BWMinMax} J. Bernstein and L. Wang, \emph{A mountain-pass theorem for asymptotically conical self-expanders}, preprint.


\bibitem{Brakke} K. Brakke, The motion of a surface by its mean curvature, Mathematical Notes 20, Princeton University Press, Princeton, NJ, 1978.

\bibitem{CGG} Y. G. Chen, Y. Giga, and S. Goto, \emph{Uniqueness and existence of viscosity solutions of generalized mean curvature flow equations}, J. Differential Geom. 33 (1991), 749--786.

\bibitem{CCMS} O. Chodosh, K. Choi, C. Mantoulidis, and F. Schulze, preprint.

\bibitem{CS} O. Chodosh and F. Schulze, \emph{Uniqueness of asymptotically conical tangent flows}, preprint. Available at \url{https://arxiv.org/abs/1901.06369}.

\bibitem{CIMW} T.H. Colding, T. Ilmanen, W.P. Minicozzi II, and B. White, \emph{The round sphere minimizes entropy among closed self-shrinkers}, J. Differential Geom. 95 (2013), no. 1, 53--69.

\bibitem{CMGen} T.H. Colding and W.P. Minicozzi II, \emph{Generic mean curvature flow I; generic singularities}, Ann. of Math. (2) 175 (2012), no. 2, 755--833.

\bibitem{EHAnn} K. Ecker and G. Huisken, \emph{Mean curvature evolution of entire graphs}, Ann. of Math. (2) 130 (1989), no. 3, 453--471.

\bibitem{EHInvent} K. Ecker and G. Huisken, \emph{Interior estimates for hypersurfaces moving by mean curvature}, Invent. Math. 105 (1991), 547--569.

\bibitem{ES1} L.C. Evans and J. Spruck, \emph{Motion of level sets by mean curvature. I}, J. Differential Geom. 33 (1991), no. 3, 635--681.

\bibitem{ES2} L.C. Evans and J. Spruck, \emph{Motion of level sets by mean curvature. II}, Trans. Amer. Math. Soc. 330 (1992), no. 1, 321--332.

\bibitem{ES3} L.C. Evans and J. Spruck, \emph{Motion of level sets by mean curvature. III}, J. Geom. Anal. 2 (1992), no. 2, 121--150.

\bibitem{ES4} L.C. Evans and J. Spruck, \emph{Motion of level sets by mean curvature. IV}, J. Geom. Anal. 5 (1995), no. 1, 77--114.

\bibitem{Huisken} G. Huisken, \emph{Asymptotic behaviour for singularities of the mean curvature flow}, J. Differential Geom. 31 (1990), no. 1, 285--299.

\bibitem{IlmanenElliptic} T. Ilmanen, Elliptic regularization and partial regularity for motion by mean curvature, Mem. Amer. Math. Soc. 108 (1994), 

\bibitem{IlmanenSing} T. Ilmanen, \emph{Singularities of mean curvature flow of surfaces}. Preprint. Available at \url{https://people.math.ethz.ch/~ilmanen/papers/sing.ps}.

\bibitem{INS} T. Ilmanen, A. Neves, and F. Schulze, \emph{On short time existence for the planar network flow}, J. Differential Geom. 111 (2019), no. 1, 39--89. 

\bibitem{KetoverZhou} D. Ketover and X. Zhou, \emph{Entropy of closed surfaces and min-max theory}, J. Differential Geom. 110 (2018), no. 1, 31--71.

\bibitem{MarquesNeves} F.C. Marques and A. Neves, \emph{Min-max theory and the Willmore conjecture}, Ann. of Math. (2) 179 (2014), no. 2, 683--782.

\bibitem{Scharlemann}  M. Scharlemann, \emph{The four-dimensional Schoenflies conjecture is true for genus two imbeddings}, Topology 23 (1984), 211--217.

\bibitem{Scharlemann2}  M. Scharlemann, \emph{Smooth spheres in $\mathbb{R}^4$ with four critical points are standard}, Invent. Math. 79 (1985), 125--141.

\bibitem{Simon} L. Simon, Lectures on geometric measure theory, Proceedings of the Centre for Mathematical Analysis, Australian National University No. 3, Canberra, 1983.

\bibitem{Stone} A. Stone, \emph{A density function and the structure of singularities of the mean curvature flow}, Calc. Var. Partial Differential Equations 2 (1994), no. 4, 443--480.

\bibitem{W} S. Wang, \emph{Round spheres are Hausdorff stable under small perturbation of entropy}, J. Reine Angew. Math, to appear. Available at \url{https://doi.org/10.1515/crelle-2017-0055}.

\bibitem{WhiteReg} B. White, \emph{A local regularity theorem for mean curvature flow}, Ann. of Math. (2) 161 (2005), no. 3, 1487--1519. 

\bibitem{JZhu} J. Zhu, \emph{On the entropy of closed hypersurfaces and singular self-shrinkers}, J. Differential Geom. 114 (2020), no. 3, 551--593.
no. 520.

\end{thebibliography}
\end{document}